\documentclass[11pt]{amsart}
\usepackage{amsmath}
\usepackage{amssymb}
\usepackage{graphicx}
\usepackage[margin=1.4in]{geometry}
\usepackage{hyperref}
\usepackage{comment}

\usepackage{amsthm}\usepackage{amsthm}\usepackage[matrix,arrow]{xy}\usepackage{enumerate}\usepackage{amsfonts}\usepackage{amsthm}\usepackage{amscd}

\newtheorem{theorem}{Theorem}
\newtheorem{lemma}[theorem]{Lemma}
\newtheorem{prop}[theorem]{Proposition}

\theoremstyle{remark}

\newtheorem{example}[theorem]{{Example}}
\newtheorem{problem}[theorem]{{Problem}}

\def\h#1{\hbox{#1}}
 \newcommand{\RR}{\mathbb{R}}
\def\bpf{\begin{proof}}
\def\epf{\end{proof}}
\def\O{\Omega}\def\L{\Lambda}
\DeclareMathOperator{\OTM}{OTM}
\def\dist{\h{dist}}

\usepackage{graphicx}
\graphicspath{%
    {converted_graphics/}
    {/}
}
\begin{document}

\title{On discontinuity of planar optimal transport maps}

\author[O. Chodosh]{Otis Chodosh}
\address{Stanford University}
\email{ochodosh@math.stanford.edu}

\author[V. Jain]{Vishesh Jain}
\email{visheshj@stanford.edu}

\author[M. Lindsey]{Michael Lindsey}
\email{lindsey3@stanford.edu}

\author[L. Panchev]{Lyuboslav Panchev}
\email{lpanchev@stanford.edu}

\author[Y. A. Rubinstein]{Yanir A.\ Rubinstein}
\address{University of Maryland}
\email{yanir@umd.edu}

\thanks{This research was supported
by the Stanford University SURIM and VPUE, and
NSF grants DGE-1147470, DMS-1206284. YAR was
also supported by a Sloan Research Fellowship.
The authors are grateful to the referee for
a very careful reading and numerous corrections, and in particular
for Proposition \ref{App4}.
}

\maketitle

\begin{abstract}
Consider two bounded domains $\Omega$ and $\Lambda$ in $\mathbb{R}^{2}$, and two sufficiently regular probability measures $\mu$ and $\nu$ supported on them. By Brenier's theorem, there exists a unique transportation map $T$ satisfying $T_\#\mu=\nu$ and minimizing the quadratic cost $\int_{\mathbb{R}^{n}}|T(x)-x|^{2}d\mu(x)$. Furthermore, by Caffarelli's regularity theory for the real Monge--Amp\`ere equations, if $\Lambda$ is convex, $T$ is continuous. 

We study the reverse problem, namely, when is $T$ discontinuous if $\Lambda$ fails to be convex? We prove a result guaranteeing the discontinuity of $T$ in terms of the geometries of $\Lambda$ and $\Omega$ in the two-dimensional case. The main idea is to use tools of convex analysis and the extrinsic geometry of $\partial\Lambda$ to distinguish between Brenier and Alexandrov weak solutions of the Monge--Amp\`ere equation. We also use this approach to give a new proof of a result due to Wolfson and Urbas.

We conclude by revisiting an example of Caffarelli, giving a detailed study of a discontinuous map between two explicit domains, and determining precisely where the discontinuities occur.
\end{abstract}

\section{Introduction}

Much work has gone into finding sufficient conditions for the optimal transportation map (OTM) to be continuous. According to Caffarelli \cite{Caf1992}, the OTM between two smooth densities 
(uniformly bounded away from zero and infinity) defined on bounded domains in $\mathbb{R}^{n}$ with smooth boundaries 
is continuous when the target domain is convex. When $n=2$,
Figalli \cite[Theorem 3.1]{Figalli} showed that even when the target domain is not convex, the OTM is still continuous outside a set of measure zero. This result has subsequently been extended to $n>2$ by Figalli and Kim \cite{FigalliKim}. These results have also been studied for Riemannian manifolds, see the recent survey by De Philippis and Figalli  \cite{DePFig}. 
However, there seems to be no known condition guaranteeing the discontinuity of planar OTMs. The main result of this article is such a condition.

In the present article we restrict ourselves to $n=2$. Throughout this article, we denote the uniform probability measures on $\Omega$ and $\Lambda$ by
\begin{equation*}
\mu:=\frac{1}{|\Omega|}1_\Omega \text{\quad and \quad} \nu:=\frac{1}{|\Lambda|}1_\Lambda,
\end{equation*}
respectively. We suppose for simplicity that
$\O$ and $\L$ have unit area, i.e., $|\O|=|\L|=1$.

Of course, one could consider more general probability measures, and certainly the results we discuss below carry over to measures with smooth densities that are uniformly bounded away from zero and from above.
Our main interest is in the following:

\begin{problem} Give conditions on $\Omega$ and $\Lambda$ guaranteeing the discontinuity of the optimal transportation map from $\mu$ to $\nu$. 
\end{problem}






The main result of this note is a sufficient condition guaranteeing the discontinuity of the OTM between two domains in $\RR^2$ assuming the source domain $\O$ is convex. This condition can be phrased solely in terms of the geodesic curvature of the boundary of the target domain. 
Moreover, we give examples to show that the numerical constant in our condition is essentially sharp. Nevertheless, we show that the condition is not a necessary one for discontinuity. In addition, we give an alternative proof of a result of Wolfson and Urbas on the nonexistence of an OTM that extends smoothly to the boundary between arbitrary domains in $\RR^2$. Our methods are different from theirs in that we rely on cyclical monotonicity. This is what allows us to prove interior discontinuity as opposed to just non-smoothness up to the boundary. 
%
Finally, we revisit an example of Caffarelli 
and analyze precisely where the discontinuities occur using symmetry arguments and results of Caffarelli and Figalli. Unlike Caffarelli, we give a constructive proof of the discontinuity, and quantify where and how this discontinuity appears. 

This note is organized as follows: In Section \ref{sec:suff-discont}, we state and prove our conditions for discontinuity. We also give several examples illustrating when these conditions do and do not hold. Then, in Section \ref{squareman}, we consider a concrete example (which we term the ``squareman'') of an optimal map between two domains in which we can precisely determine how the map fails to be continuous. Finally, in Appendix \ref{sec:app}, we state and prove several lemmas that we need for the proof of the curvature condition.

\section{A sufficient condition for discontinuity}\label{sec:suff-discont}

In this section we derive a sufficient condition for the discontinuity of the OTM between $\O$ and $\L$ 
based on the geometry of the boundaries. We further show how our method proves a result of Wolfson, which was subsequently refined by Urbas. It is interesting to note that while Wolfson's original proof uses symplectic geometry, our approach is based on convex analysis.

Consider a simple closed $C^{2}$ curve $C\subset\mathbb{R}^{2}$,
and let ${\bf n}$ denote the inward-pointing unit normal along $C$.
Given a unit-speed parametrization $\gamma:I\to\mathbb{R}^{2}$ of $C$
(here, $I\subset\mathbb{R}$ denotes an interval, which we can assume
equals $[0,L]$ without any loss of generality), the curvature of $C$
is defined to be the function $\kappa:C\to\mathbb{R}$ satisfying
$\gamma''=\kappa{\bf n}$. Note that since we have defined the signed curvature with respect to the inward pointing unit normal, it is independent of the orientation of the curve. 

In this article, we will refer numerous times to connected subsets
or connected components of a simple (possibly closed) curve. Both
of these simply refer to a subset of the (image of the) curve which
is connected (and hence, path connected) in the subspace topology
induced from $\mathbb{R}^{2}$. In particular, we are not referring
to maximally (path) connected components of the curve. Since a continuous
bijection from a compact space to a Hausdorff space is automatically
a homeomorphism, and all our curves have domain $[0,1]$ or $S^{1}$, it follows
that a subset of the curve $\gamma$
is connected if and only if it is of the form $\gamma(I)$, where
$I$ is a sub-interval of $[0,1]$ or $S^{1}$. 

\begin{theorem}
\label{mainthm}
Let $\Omega$ and
$\Lambda$ be bounded, connected, simply connected open domains in $\mathbb{R}^{2}$
such that $\partial\Omega$ and $\partial\Lambda$ are $C^{3}$, closed
curves. Assume $\Omega$ is convex. 
Equip $\Omega$ and $\Lambda$ with the uniform measures $\mu$
and $\nu$. Let $\kappa_{\partial\Omega}$ and $\kappa_{\partial\Lambda}$
be the signed curvatures of $\partial\Omega$ and $\partial\Lambda$
with respect to the corresponding inward-pointing unit normal fields.
If there exists a connected subset $J\subset\partial\Lambda$ with
\begin{equation}
\label{CurvEq}
{\displaystyle \int_{J}\kappa_{\partial\Lambda}<-\pi},
\end{equation}
then $T_1,$
the OTM from $\Omega$ to $\Lambda$, is discontinuous. 

\end{theorem}

We employ similar techniques, together with an additional modification, to give a new proof of the following result due to Wolfson and Urbas \cite{Wolfson,Urbas2007}.

\begin{theorem}[Wolfson and Urbas] \label{maincor}
Let $\Omega$ and $\Lambda$ be two bounded, connected, simply connected domains
in $\mathbb{R}^{2}$ with $C^{2}$ boundaries. Let $\kappa_{\partial\Omega}$
and $\kappa_{\partial\Lambda}$ denote the signed curvatures (as defined above) of the two boundaries.
Assume that 
\begin{equation}\label{eq:1}
\inf_{J\subset\partial\Lambda}\int_{J}\kappa_{\partial\Lambda}
\le
\inf_{I\subset\partial\Omega}\int_{I}\kappa_{\partial\Omega}-\pi,
\end{equation}
where $I$ and $J$ are connected subsets of $\partial\Omega$
and $\partial\Lambda$, respectively. Then there does not exist a $C^1$-diffeomorphism $T_1:\overline{\Omega}\to\overline{\Lambda}$ whose
restriction to $\O$ is an OTM.
\end{theorem}

When $\Omega$ is convex, of course $\kappa_{\partial\Omega}\ge0$. One may
ask whether \eqref{CurvEq} may be weakened. 
Below, we will construct an example (Example 6) to show that at least when the $C^{3}$ hypothesis in Theorem \ref{mainthm} is replaced with piecewise smooth, the constant $-\pi$ in \eqref{CurvEq} cannot be increased. 
However, \eqref{CurvEq} is not a necessary condition: in Section \ref{squareman} we will construct an 
example where $\Omega$ is convex and $\Lambda$ has a connected subset of total curvature
of at most $-\frac{\pi}{2}$, but $\OTM(\Omega,\Lambda)$ is discontinuous.

Condition \eqref{eq:1} is also not necessary: below (Example 5), we construct $\Omega$ and $\Lambda$ so that $\inf_{I\subset\partial\Lambda}\int_{I}\kappa_{\partial\Lambda}
=\inf_{J\subset\partial\Omega}\int_{J}\kappa_{\partial\Omega}$,
but $\OTM(\Omega,\Lambda)$ is not a $C^1$ diffeomorphism up to the boundary.

The strength of Theorem \ref{mainthm} lies in the fact
that it does not assume any nice behavior of the optimal map near
the boundary. At the same time it shows not only lack of regularity,
but discontinuity. In the proof of Theorem \ref{mainthm}, the convexity 
is used to show a continuous
OTM from $\Omega$ to $\Lambda$ is necessarily a $C^1$-diffeomorphism, by Caffarelli's regularity theorem. If we are concerned only with the
nonexistence of OTMs which are $C^1$ diffeomorphisms up to the
boundary, one can do away with the convexity assumption, which
is the content of Theorem \ref{maincor}.

The proof of Theorem \ref{maincor} contains two differences
from the proof of Theorem \ref{mainthm}.
First, the technical Lemma \ref{technicallemma} is no longer necessary, since
we are assuming regularity up to the boundary.
On the other hand, condition \eqref{eq:1} is weaker
than \eqref{CurvEq}, and so one must make use of cyclical
monotonicity and not just of monotonicity.


\begin{proof}[Proof of Theorem \ref{mainthm}]
Assume for the  sake of contradiction that the OTM 
$$
T_1:\Omega\to\Lambda
$$
is continuous everywhere on $\Omega$. Since $\Omega$ is convex,
we have from Caffarelli's regularity theorem 
\cite[Theorem 12.50]{Villani:oldandnew}
that the map $T_2$ (the optimal map from $\Lambda$ to $\Omega$) is 
$C^2$ everywhere on $\Lambda$.
We also know that for $\mu$-almost all $x$ and for $\nu$-almost
all $y$, $T_2\circ T_1(x)=x$ and $T_1\circ T_2(y)=y$
\cite[Theorem 2.12]{Villani}. Since both compositions are continuous
and are equal to the identity almost everywhere, it follows that they
must be the identity everywhere, and therefore, that $T_2^{-1}=T_1$
everywhere. Moreover, since $T_2$ is $C^2$ by Caffarelli's regularity
theorem, and since its Jacobian matrix is nonsingular at every point
in $\Lambda$ by the Monge-Amp\`ere equation, it follows from the inverse
function theorem that $T_1$ is also $C^2$. Hence $T_2$
is a $C^2$-diffeomorphism between $\Lambda$ and $\Omega$. 

For sufficiently small $\epsilon>0$, consider the sets 
(see Figure \ref{fig:1})
$$
\Lambda_{\epsilon}=\{x\in\Lambda:
\h{dist}(x,\partial\Lambda)<\epsilon\},\qquad
\Gamma_{\epsilon}=\{x\in\Lambda:
\h{dist}(x,\partial\Lambda)=\epsilon\}.
$$ 
From Proposition \ref{App1}, we know that there exists $\hat{\epsilon}>0$ such that for
every $0<\epsilon\leq\hat{\epsilon}$, the curve $\Gamma_{\epsilon}$
is $C^{1}$. Furthermore, there exists a diffeomorphism $f_{\epsilon}:\partial\Lambda\to\Gamma_{\epsilon}$
such that the vector $f_{\epsilon}(x)-x$ is normal to $\partial\Lambda$
at $x$ and to $\Gamma_{\epsilon}$ at $f_{\epsilon}(x)$, and has
magnitude $|f_{\epsilon}(x)-x|=\epsilon$. Consider the image $T_2(\Gamma_{\epsilon})=\Theta_{\epsilon}$.


\begin{figure}
\includegraphics[scale=.8]{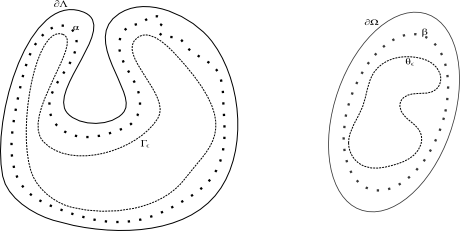}
\caption{$\alpha=T_1(\beta)$ is ``close'' to $\partial\Lambda$ and $\beta$ is convex.}\label{fig:1}
\end{figure}

Since $T_2$ is $C^2$ and $\Gamma_{\epsilon}$ is compact, we
have that $\Theta_{\epsilon}\subset\Omega$ is also a compact, closed
$C^{1}$ curve, so that in particular, $\dist(\partial\Omega,\Theta_{\epsilon})>\delta>0$.
In order to be able to work in the interior of $\Omega$, we would
like to construct a $C^{2}$ convex curve $\beta\subset\Omega$ such
that the interior of the region enclosed by $\beta$ completely contains
$\Theta_{\epsilon}$. 
Note that here, and elsewhere, we use the standard terminology of calling a closed curve convex if it is the boundary of a bounded convex set. 
Such a $\beta$ can readily be constructed:
we pick a point $x_{0}$ contained in the interior of the region bounded
by $\Theta_{\epsilon}$ and scale points on $\partial\Omega$ with
respect to $x_{0}$ by a factor $1-\tilde{\epsilon}<t<1$ for a small enough $\tilde{\epsilon}>0$. Denote
the resulting curve by $\beta_{t}$. Then, $\beta_{t}$ is seen to
be convex and $C^{2}$, since $\partial\Omega$ is convex and $C^{2}$ (in fact we have assumed that it is $C^{3}$).
Further, we can always arrange $\dist(\partial\Omega,\beta_{t})<\frac{\delta}{2}$
by picking $\tilde{\epsilon}>0$ small enough. In particular, we can
choose $t$ so that $\beta=\beta_{t}$ contains $\Theta_{\epsilon}$
completely in its interior, and is as close to $\partial\Omega$ as
desired. Note that the convexity of $\beta$ implies that the signed
curvature $\kappa_{\beta}$ of $\beta$ with respect to the inward
unit normal field is non-negative. 

Next, consider the $C^{2}$ curve $\alpha=T_{1}(\beta)$. We claim
that $\alpha$ contains $\Gamma_{\epsilon}$ in its interior in the
sense that every continuous path between $\Gamma_{\epsilon}$ and
$\partial\Lambda$ must intersect $\alpha$. Indeed, let $c:[0,1]\rightarrow\bar{\Lambda}$ (note that we may assume without loss of generality that $c([0,1))\subset\Lambda$)
be a continuous map such that $c(0)\in\Gamma_{\epsilon}$ and $c(1)\in\partial\Lambda$.
Then, $T_{2}(c(0))\in\Theta_{\epsilon}$, while $\cap_{t\in(0,1)}\overline{T_{2}(c(t,1))}\neq\emptyset$
by the finite intersection property applied to the compact space $\overline{\Omega}$.
We claim that $\cap_{t\in(0,1)}\overline{T_{2}(c(t,1))}\subset\partial\Omega$.
Indeed, suppose that $\cap_{t\in(0,1)}\overline{T_{2}(c(t,1))}\nsubseteq\partial\Omega$.
Since the intersection is nonempty and contained in $\overline{\Omega}$,
we must have some point $q\in\Omega$ such that $q\in\cap_{t\in(0,1)}\overline{T_{2}(c(t,1))}$.
Consider the point $p=T_{1}(q)$. Since $p\in\Lambda$, we can (by
the continuity of $c$) find some $t'\in(0,1)$ sufficiently close
to $1$ such that $p\notin c[t',1]$. Since $T_{2}$ is an open map
on $\Lambda$ by hypothesis, it follows that $q=T_{2}(p)\notin\overline{T_{2}(c(t',1))}$,
which gives us a contradiction. Once we have that $\cap_{t\in(0,1)}\overline{T_{2}(c(t,1))}\subset\partial\Omega,$we
can prove that $T_{2}(c)$ intersects $\beta$, and hence, that $c$
intersects $\alpha$. For this, it clearly suffices to show that
there exists some $t_{0}\in(0,1)$ such that for all $t\in(t_{0},1)$
one has $\dist(T_{2}(c(t)),\partial\Omega)<\dist(\beta,\partial\Omega)$.
Suppose such a $t_{0}$ does not exist. Then, there exists a sequence
$t_{n}\uparrow1$ such that $p_{n}=T_{2}(c(t_{n}))$ sits inside the
closure of the domain bounded by $\beta$ (we will denote this domain
by dom$(\beta)$). By compactness of dom$(\beta)$, we can assume after
possibly passing to a subsequence that $p_{n}\rightarrow p$ in dom$(\beta)$.
But then, we have that $p\in\Omega$, and $p\in\cap_{t\in(0,1)}\overline{T_{2}(c(t,1))}$,
which is a contradiction. Note that since connected components are
preserved under homeomorphisms, $\Gamma_{\epsilon}$ is a Jordan curve,
and $\alpha$ contains at least one point in $\Gamma_{\epsilon}$,
we have also showed that $\alpha\subset\Gamma_{\epsilon}$. 

We will need the following lemma.

\begin{lemma}\label{technicallemma} Let $\kappa_{\alpha}$ be the signed curvature
of $\alpha$ with respect to the inward pointing unit normal vector
field. Then for sufficiently small $\epsilon$ there exists some connected
subset $I_{1}\subset\alpha$ for which the total signed curvature
is less than $-\pi$, i.e., ${ \int_{I_{1}}\kappa_{\alpha}<-\pi}$.
\end{lemma}

The lemma says that a closed
curve ``close'' to the boundary of our domain must exhibit similar
curvature behaviour. We will prove this after we
show how it implies the curvature condition.

Using Lemma \ref{technicallemma}, we can pick a connected $I_{1}\subset\alpha$ such
that${ \int_{I_{1}}\kappa_{\alpha}<-\pi}$. 
Let $|I_{1}|=l>0$
be the length of $I_{1}$. We denote a unit speed parametrization
for $I_{1}$ by $\alpha_{1}:[0,l]\to I_{1}$. Since $T_2$ is
an optimal map from $\nu$ to $\mu$, we have from monotonicity
\cite[Proposition 2.24]{Villani} that for any two points $x,y\in\alpha$,
$\langle x-y,T_2(x)-T_2(y)\rangle\geq0$. In particular, for
every $t\in[0,l)$ and for a suitably small $h>0$, we must have that
\[
\langle\alpha_{1}(t+h)-\alpha_{1}(t),T_2(\alpha_{1}(t+h))-T_2(\alpha_{1}(t))\rangle\geq0
.\]
 Dividing by $h^{2}$ and letting $h\to0$ in the previous equation, we get that 
\begin{equation}\label{2}
\langle\dot{\alpha_{1}}(t),\dot{\beta_{1}}(t)\rangle\ge0
\end{equation}
where $\beta_{1}=T_2\circ\alpha_{1}$ is a parametrization of
$T_{2}(I_{1})$. Note ${\dot\beta_{1}(t)\neq0}$ as $T_2$ is locally
a diffeomorphism around every $\alpha_{1}(t)\in\Lambda$. 

Equation \eqref{2} implies that the tangent vectors to $\alpha$
and $\beta$ at any $x\in\alpha$ and $T_2(x)\in\beta$ must have
a non-negative inner product. We show that this cannot happen, thereby
proving Theorem \ref{mainthm} (modulo the proof of Lemma \ref{technicallemma}). Intuitively,
it is clear that this cannot happen: as we move along $I_{1}$ counterclockwise, the
tangent vector at a point along $I_{1}$ rotates clockwise, while the tangent vector at the corresponding point on $\beta$ rotates
counterclockwise. Monotonicity dictates that the angle between the
corresponding vectors must always be within $\frac{\pi}{2}$. However,
since the curvature of $I_{1}$ is less than $-\pi$, and the curvature
of the corresponding connected subset of $\beta$ is $\geq0$ by
convexity, the angle between corresponding tangent vectors changes
by more than $-\pi$ when traversing $I_{1}$, and therefore, cannot
lie in $[-\frac{\pi}{2},\frac{\pi}{2}]$ at all points of $I_{1}$. 

To make the above discussion more precise, consider the ``tail-to-tail''
angle between two vectors. This is the standard notion of angle which takes
values in the interval $(-\pi,\pi]$. Given an ordered pair of vectors,
we define the angle between them as the signed ``tail-to-tail angle''
between them, with the sign taken to be positive if we move counterclockwise
from the first vector to the second and negative otherwise. We now
define $f:[0,l]\to(-\pi,\pi]$, where $f(t)$ is the angle from $\dot{\beta_{1}}(t)$ to $\dot{\alpha_{1}}(t)$. From \eqref{2}
it follows that $f(t)\in[-\frac{\pi}{2},\frac{\pi}{2}]$ for every
$t\in[0,l]$. Set $J_{t}=\alpha_{1}([0,t])\subset I_{1}$. Then 
\begin{equation}
f(t)-f(0)=\int_{J_{t}}\kappa_{\alpha}-\int_{T_2(J_{t})}\kappa_{\beta}+2k(t)\pi\label{3}
\end{equation}
 where $k(t)\in\mathbb{Z}$.

Since $\alpha$ and $\beta$ are $C^{2}$, the unsigned angle between
$\dot{\alpha_{1}}(t)$ and $\dot{\beta_{1}}(t)$ defined from $[0,l]$
to $[0,\infty)$ varies continuously. Therefore, any discontinuities
in the signed angle $f(t)$ can occur only near the values $-\pi$
and $\pi$. But since $f(t)\in[-\frac{\pi}{2},\frac{\pi}{2}]$, it
is never close to $\pi$ or $-\pi$ and therefore must be continuous
everywhere on $[0,l]$. Hence $2k(t)\pi$ must also be continuous.
But $k(t)$ is integer-valued, so it is the constant $k(0)=0$. In
particular \eqref{3} yields 
\begin{equation}
f(t)=\int_{J_{t}}\kappa_{\alpha}-\int_{T_2(J_{t})}\kappa_{\beta}+f(0).\label{eq:4}
\end{equation}
 Setting $t=l$ in \eqref{eq:4}, we get 
\begin{eqnarray}
f(l)-f(0)  =  \int_{J_{l}}\kappa_{\alpha}-\int_{T_2(J_{l})}\kappa_{\beta}\label{eq:5}
 = \int_{I_{1}}\kappa_{\alpha}-\int_{T_2(I_{1})}\kappa_{\beta}\nonumber 
  <  -\pi+0 =  -\pi,\nonumber 
\end{eqnarray}
 where the inequality holds since $\int_{I_{1}}\kappa_{\alpha}(x)dx<-\pi$
and $\int_{J}\kappa_{\beta}(x)dx\geq0$ for every $J\subset\beta$
as $\beta$ is the boundary of a convex set. On the other hand $f(t)-f(0)\in[-\pi,\pi]$
since $f(t)\in[-\frac{\pi}{2},\frac{\pi}{2}]$ for all $t\in[0,l]$.
This contradicts $f(l)-f(0)<-\pi$. Hence our assumption that $T_1$
is continuous must be incorrect, and $T_1$ is discontinuous as
desired. 
\end{proof}
\begin{proof}[Proof of Lemma \ref{technicallemma}]
Recall $\Gamma_{\epsilon}=\{x\in\Lambda:$ dist$(x,\partial\Lambda)=\epsilon\}$,
and $\alpha$ contains $\Gamma_{\epsilon}$ in its interior in the
sense that every continuous path between $\Gamma_{\epsilon}$ and
$\partial\Lambda$ must intersect $\alpha$. Also recall that by assumption,
${ \int_{I}\kappa_{\partial\Lambda}<-\pi}$ for some connected
subset $I\subset\partial\Lambda$. 
The underlying idea is to choose a connected
subset of $\alpha$ which is close to the connected subset $I$
of $\partial\Lambda$, and then show that this subset must necessarily
contain a further connected subset of signed curvature less than
$-\pi$. We have illustrated the arguments made below in Figure \ref{fig:2}.

Let $A$ and $B$ be the endpoints of $I$. Let $\overline{\epsilon}>0$.
Then pick $C\in I$ close to $A$ so that the angle between $AC$
and the tangent to $\partial\Lambda$ at $A$ is less than $\frac{\overline{\epsilon}}{2}$.
Similarly pick some $D\in I$ close to $B$ which satisfies the same
criterion. For $\epsilon>0$ small, consider the curve $\Gamma_{\epsilon}$.
Recall $f_{\epsilon}:\Gamma\to\Gamma_{\epsilon}$ maps points on $\Gamma$
to points $\epsilon$ away on $\Gamma_{\epsilon}$. Now $A_{2}=f_{\epsilon}(A)$, $B_{2}=f_{\epsilon}(B)$, $C_{2}=f_{\epsilon}(C)$ and $D_{2}=f_{\epsilon}(D)$
are points on $\Gamma_{\epsilon}$. By selecting $\epsilon>0$ small
enough, we can ensure that for any $X\in AA_{2}$ and any $Y\in CC_{2}$
, the angle between $AA_{2}$ and $XY$ belongs to the interval $(\frac{\pi}{2}-\overline{\epsilon},\frac{\pi}{2}+\overline{\epsilon})$
, and also that for any $X_{1}\in BB_{2}$ and $Y_{1}\in DD_{2}$,
the angle between $BB_{2}$ and $X_{1}Y_{1}$ belongs to the interval
$(\frac{\pi}{2}-\overline{\epsilon},\frac{\pi}{2}+\overline{\epsilon})$.
This is equivalent to saying that the direction of $XY$ differs by
no more than $\overline{\epsilon}$ from the direction of $\partial\Lambda$
at the point $A$ , and the direction of $X_{1}Y_{1}$ differs by
no more than $\overline{\epsilon}$ from the direction of $\partial\Lambda$
at the point $B$ . 

Next, from Proposition \ref{App3}, we have that there exists
a connected subset $\tilde{I}\subset\alpha$ which is contained in the region
bounded by $I$, $AA_{2}$, $BB_{2}$ and $\Gamma_{\epsilon}$ , and
which has endpoints $A_{1}\in AA_{2}$ and $B_{1}\in BB_{2}$. We
move along $\tilde{I}$ from $A_{1}$ to $B_{1}$ and denote by $C_{1}$
the point where $\tilde{I}$ intersects $CC_{2}$ for the first time.
We denote by $D_{1}$ the point where $\tilde{I}$ intersects $DD_{2}$
for the last time. From Proposition \ref{App2}, we further
know that there exists a point $E_{1}$ lying on the portion of
$\tilde{I}$ between $A_{1}$ and $C_{1}$ at which the direction
of the tangent to $\tilde{I}$ coincides with the direction of $A_{1}C_{1}$.
In particular, the angle between the tangent to $\tilde{I}$ at $E_{1}$
and the segment $AA_{2}$ is in the interval $(\frac{\pi}{2}-\overline{\epsilon},\frac{\pi}{2}+\overline{\epsilon})$. Similarly, we can choose a point $F_{1}$ that lies on the portion
of $\tilde{I}$ connecting $D_{1}$ to $B_{1}$ such that the angle
between the tangent to $\tilde{I}$ at $F_{1}$ and $BB_{2}$ is in
the interval $(\frac{\pi}{2}-\overline{\epsilon},\frac{\pi}{2}+\overline{\epsilon})$.

\begin{figure}
\includegraphics{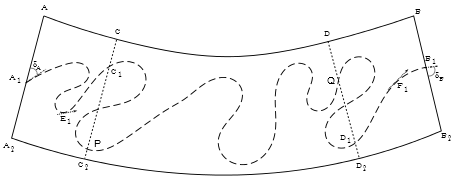}
\caption{The curve $A_{1}B_{1}\subset\alpha$ enclosed by the region of negative curvature.}
\label{fig:2}
\end{figure}

Finally, we are in a position to establish the existence of the interval
$I_{1}\subset\tilde{I}$ for which ${ \int_{I_{1}}\kappa_{\alpha}<-\pi}$ with respect to the unit normal pointing inside the bounded component of the complement of the Jordan curve $\alpha$. Equivalently, it is the signed curvature with respect to the standard orientation on $\mathbb{R}^{2}$ and with $I_{1}$ oriented from $B_{1}$ to $A_{1}$. We will always use this sign of curvature for (connected components) of $\tilde{I}$. 
Denote by $\delta_{A}$ the angle between the tangent to $\tilde{I}$
at the point $A_{1}$ and the vector $A_{1}A$ and by $\delta_{B}$, the angle between the vector $B_{1}B_{2}$ and the tangent to $\tilde{I}$. In this definition, we have used the tangent vector to $\tilde{I}$ when it is oriented from $A_1$ to $B_1$.  Note that both $\delta_{A}$ and $\delta_{B}$ are in $[0,\pi]$. We now apply the Gauss-Bonnet theorem to the
region (with its boundary oriented counterclockwise) bounded by $I$ , $\tilde{I}$ , $AA_{1}$ and $BB_{1}$  to
get that 
\begin{equation}\label{eq:6}
\int_{I}\kappa_{\partial\Lambda}+\frac{\pi}{2}+(\pi-\delta_{A})-\int_{\tilde{I}}\kappa_{\alpha}+(\pi-\delta_{B})+\frac{\pi}{2}=2\pi.
\end{equation}

Note the negative sign in front of the integral over $\tilde{I}$, which comes from the fact that the orientation of $\tilde{I}$ in this calculation
is from $A_1$ to $B_1$, which is the opposite of what we had originally used (i.e. from $B_1$ to $A_1$) in computing $\int_{\tilde{I}}\kappa_{\alpha}$. Simplifying \eqref{eq:6}, we have $\int_{\tilde{I}}\kappa_{\alpha}=\int_{I}\kappa_{\partial\Lambda}+\pi-\delta_{B}-\delta_{A}$
We will now split $\tilde{I}$ into three parts and show that some
connected combination of these three parts has a total signed curvature
lesser than $-\pi$ with the original orientation i.e. with $\tilde{I}$
going from $B_{1}$ to $A_{1}$. Note that the points $E_{1}$ and
$F_{1}$ provide such a splitting naturally. Denote these three components of $\tilde{I}$ between
$A_{1}$ and $E_{1}$, $E_{1}$ and $F_{1}$, and $F_{1}$ and $B_{1}$
by $\tilde{I}_{1}$, $\tilde{I}_{2}$, $\tilde{I}_{3}$ respectively.
For some $\hat{\epsilon}_{1},\hat{\epsilon}_{3}\in(-\overline{\epsilon},\overline{\epsilon})$,
\[
\int_{\tilde{I}_{1}}\kappa_{\alpha}=\left(\frac{\pi}{2}+\hat{\epsilon}_{1}\right)-\delta_{A}+2k\pi
\]
 for some integer $k$, and 
\begin{equation*}
\int_{\tilde{I}_{3}}\kappa_{\alpha} = (\pi-\delta_{B}) - \left(\frac{\pi}{2}-\hat{\epsilon}_{3}\right)+2k'\pi
 = \frac{\pi}{2}-\delta_{B}+2k'\pi+\hat{\epsilon}_{3}
\end{equation*}
 for some integer $k'$.

If $k<0$, then 
\begin{equation*}
\int_{\tilde{I}_{1}}\kappa_{\alpha}<-\frac{3\pi}{2}+\overline{\epsilon}.
\end{equation*}
Taking $\overline\epsilon$ small enough, we get
\begin{equation*}
 \int_{\tilde{I}_{1}}\kappa_{\alpha}<-\pi
\end{equation*}
in which case $\tilde{I_{1}}$ is the desired segment.

If $k>0$, then 
\begin{align*}
\int_{\tilde{I}\backslash\tilde{I}_{1}}\kappa_{\alpha} & =  \int_{I}\kappa_{\partial\Lambda}+\pi-\delta_{B}-\delta_{A}-\int_{\tilde{I}_{1}}\kappa_{\alpha}\\
 & <  -\pi+\pi-\delta_{B}-\delta_{A}-\frac{\mbox{\ensuremath{\pi}}}{2}+\delta_{A}-2\pi+\overline{\epsilon}\\
 & = -\delta_{B}-\frac{5}{2}\pi+\overline{\epsilon} < -\pi,
\end{align*}
 as long as we take $\overline{\epsilon}>0$ sufficiently small. Hence,
in this case we have that $\tilde{I}\backslash\tilde{I}_{1}=\tilde{I}_{2}\cup\tilde{I}_{3}$
has total curvature less than $-\pi$.

Similarly, if $k'\neq0$, then using one of the arguments above,
we can take $I_{1}$ to be either $\tilde{I}_{3}$ or $\tilde{I}_{1}\cup\tilde{I}_{2}$.

Thus, the only case left to investigate is when $k=k'=0$. In this
case, 
\begin{equation*}
\int_{\tilde{I}_{1}}\kappa_{\alpha}=\frac{\pi}{2}-\delta_{A}+\hat{\epsilon}_{1} \text{\qquad and \qquad} \int_{\tilde{I}_{3}}\kappa_{\alpha}=\frac{\pi}{2}-\delta_{B}+\hat{\epsilon}_{3}.
\end{equation*}
Combining these two equations, we get that 
\begin{align*}
\int_{\tilde{I}_{2}}\kappa_{\alpha} & =  \int_{\tilde{I}}\kappa_{\alpha}-\int_{\tilde{I}_{1}}\kappa_{\alpha}-\int_{\tilde{I}_{3}}\kappa_{\alpha}\\
& =\int_{I}\kappa_{\partial\Lambda}+\pi-\delta_{B}-\delta_{A}-\frac{\pi}{2}+\delta_{A}-\frac{\pi}{2}+\delta_{B}-\hat{\epsilon}_{1}-\hat{\epsilon}_{3}\\
 & <  \int_{I}\kappa_{\partial\Lambda}+2\overline\epsilon.
\end{align*}
In particular, if we choose $\overline{\epsilon}<-\frac{\int_{I}\kappa_{\partial\Lambda}+\pi}{2}$, then, it follows that $\int_{\tilde{I}_{2}}\kappa_{\alpha}\le\int_{I}\kappa_{\partial\Lambda}+2\overline{\epsilon}<-\pi$
and thus in this case, $\tilde{I}_{2}$ satisfies the claim. This
completes the proof of the lemma.
\end{proof}

\begin{proof}[Proof of Theorem \ref{maincor}] 
Assume the existence of such $T_1$. 
We follow the notation established in the proof of Theorem \ref{mainthm}. 
Pick $J\subset\partial\Lambda$
so that $\int_{J}\kappa_{\partial\Lambda}\le\inf_{I\subset\partial\Omega}\int_{I}\kappa_{\partial\Omega}-\pi$. We can choose such a $J\subset\partial\Lambda$ because by assumption, $\inf_{J\subset\partial\Lambda}\int_{J}\kappa_{\partial\Lambda}\le\inf_{I\subset\partial\Omega}\int_{I}\kappa_{\partial\Omega}-\pi$, and the infimum on the left hand side is attained because $\partial\Lambda$ is compact.
Let $\gamma:[0,l]\to J$ be the unit speed parametrization of $J$,
so that $\gamma''(t)=\kappa_{\partial\Lambda}(\gamma(t))\textbf{n}_{\partial\Lambda}(\gamma(t))$ where $\textbf{n}_{\partial\Lambda}(\gamma(t))$ is the inward pointing unit normal at $\gamma(t)$ and $l$ is the
length of $J$.
From monotonicity (recall \eqref{2}), $f(t)\in[-\frac{\pi}{2},\frac{\pi}{2}]$
for all $0\leq t\leq l$.
But we also have that for every $t\in[0,l]$,
\[
f(t)-f(0)=\int_{\gamma([0,t])}\kappa_{\partial\Lambda}-\int_{T_1^{-1}(\gamma([0,t]))}\kappa_{\partial\Omega}+2k(t)\pi
\]
where as before $k(t)$ is an integer that accounts for the discontinuity that amounts from restricting the codomain of $f$
to $(-\pi,\pi]$. Since $f(t)\in[-\frac{\pi}{2},\frac{\pi}{2}]$,
$k = 0$. Thus
\[
\int_{\gamma([0,t])}\kappa_{\partial\Lambda}-\int_{T_1^{-1}(\gamma([0,t]))}\kappa_{\partial\Omega}=f(t)-f(0).
\]
On the other hand, from the choice of $J\subset\partial\Lambda$,
we have that
\begin{eqnarray*}
\int_{T_1^{-1}(\gamma([0,l]))}\kappa_{\partial\Omega}-\int_{\gamma([0,l])}\kappa_{\partial\Lambda} =  \int_{T_1^{-1}(J)}\kappa_{\partial\Omega}-\int_{J}\kappa_{\partial\Lambda}
  \ge \inf_{I\subset\partial\Omega}\int_{I}\kappa_{\partial\Omega}-\int_{J}\kappa_{\partial\Lambda}\ge \pi.
\end{eqnarray*}

To conclude the proof, we claim that
$f(t)\in(-\frac{\pi}{2},\frac{\pi}{2})$. To see this, observe that
for any $x,y,z\in\overline{\L}$
$$
\langle x,T_2(x)-T_2(y)\rangle+
\langle y,T_2(y)-T_2(z)\rangle+
\langle z,T_2(z)-T_2(x)\rangle\ge0,
$$
by cyclical monotonicity, or equivalently
$$
\langle z-y,T_2(z)-T_2(y)\rangle \ge
\langle x-z,T_2(y)-T_2(x)\rangle.
$$
Now by Brenier's theorem $T_2=\nabla w$ for some convex
function $w$ on $\L$ which is $C^2$ by our assumptions, and since $\det\nabla^2 w=1$ this
function is strongly convex in the sense that $\nabla^2 w>0$.
Since $T_2$ is $C^1$ we have
$T_2(y)-T_2(x)=DT_2(x)\cdot (y-x)+v(y)$,
where $|v(y)|=o(|x-y|)$. Thus,
$$
\langle z-y,T_2(z)-T_2(y)\rangle\ge
\langle x-z,\nabla^2w(x)\cdot(y-x)+v(y)\rangle.
$$
Now pick a closed disk that is contained in $\overline{\Lambda}$ s.t. 
the circle bounding the disk is tangent to $\partial\Lambda$ at $x$. Let 
$\tilde{\gamma}$ be a unit speed parametrization of this circle. In 
particular, for some fixed $t_2$ we have that $\tilde{\gamma}(t_2)  = x$. Set $y=\tilde{\gamma}(t_1)$ 
and $z = \tilde{\gamma}(t_3)$ and such that $|y-x|=|z-x|$ or equivalently $2t_2=t_1+t_3$.
Also, set $\delta(t):=T_2\circ\tilde{\gamma}(t)$.
Then,
$$
\langle \tilde{\gamma}(t_3)-\tilde{\gamma}(t_1),
\delta(t_3)-\delta(t_1)\rangle\ge
\langle x-z,\nabla^2w(x)\cdot(y-x)+v(y)\rangle.
$$
Note that for sufficiently small $t_3 - t_1$, there exists a constant $C>0$ such that
$\frac1C|z-y|<t_3-t_1<C|z-y|$ and $|x-y|\le|z-y|\le2|x-y|$.
Thus dividing
both sides of the equation by $(t_3-t_1)^2$ and
taking the limit as $t_3-t_1$ tends to zero gives
$$
\langle\dot{\tilde\gamma}(t_2),\dot\delta(t_2)\rangle
\ge C' \nabla^2 w\langle\nu,\nu\rangle>0,
$$
as $|\nu|=1$ is the tangent vector to the disk at $x$ and $C'>0$. This completes the proof of the theorem.
\end{proof}

\subsection{Examples}\label{examples}

In this section we will explore several concrete examples. The
first one shows that the extended curvature criterion for the non-existence
of OTMs which are diffeomorphisms up to the boundary is sufficient
but not necessary. The second example shows that it is reasonable not to
hope for a constant better than $-\pi$ in the curvature criterion.  

\begin{example}
Consider the two domains pictured in Figure \ref{fig:3} and suppose there exists an optimal map $T:\overline{\Omega}\to\overline{\Lambda}$,
which is a $C^1$-diffeomorphism. Given $\epsilon>0$ small, we construct
our domains so that $-2\pi\leq\inf_{I\subset\partial\Omega}\int_{I}\kappa_{\partial\Omega}=\inf_{J\subset\partial\Lambda}\int_{J}\kappa_{\partial\Lambda}<-2\pi+\epsilon$.
In particular, the hypotheses of Theorem \ref{maincor} (and, of course, those of Theorem \ref{mainthm}, since $\Omega$ is not convex) 
do not apply here. 

Note that $\partial\Lambda$ has 4 disjoint connected subsets $\{J_{1},\dots,J_{4}\}$
of signed curvature $-2\pi+\tilde{\epsilon}$ for some $0<\tilde{\epsilon}<\epsilon$,
while $\partial\Omega$ has only one negatively curved component,
with signed curvature no lesser than $-2\pi$. In particular, the
integral of $\kappa_{\partial\Omega}$ over any union of connected
subsets of $\partial\Omega$ cannot be lesser than $-2\pi$. 

As before, we show that an OTM $T$ (as above) cannot exist, by showing
that it violates the monotonicity condition for OTMs. Indeed, by the
``tail-to-tail'' argument used in the proof of Theorem 2, we have
that 
\[
-\pi\leq\int_{J_{i}}\kappa_{\partial\Lambda}-\int_{T(J_{i})}\kappa_{\partial\Omega}<-2\pi+\tilde{\epsilon}-\int_{T(J_{i})}\kappa_{\partial\Omega}
\]

so that 
\[
\int_{T(J_{i})}\kappa_{\partial\Omega}<-\pi+\tilde{\epsilon}<-\pi+\epsilon
\]

Note that the images $T(J_{i})$ are disjoint. Therefore, from the
above discussion, we have 
\[
-2\pi\leq\int_{\cup T(J_{i})}\kappa_{\partial\Omega}=\sum_{i=1}^{4}\int_{T(J_{i})}\kappa_{\partial\Omega}<-4\pi+4\epsilon<-2\pi
\]

for $\epsilon$ small enough, which is a contradiction. This completes
the proof of the non-optimality of $T$. 

 
\end{example}



\begin{figure}
\includegraphics[scale=0.35]{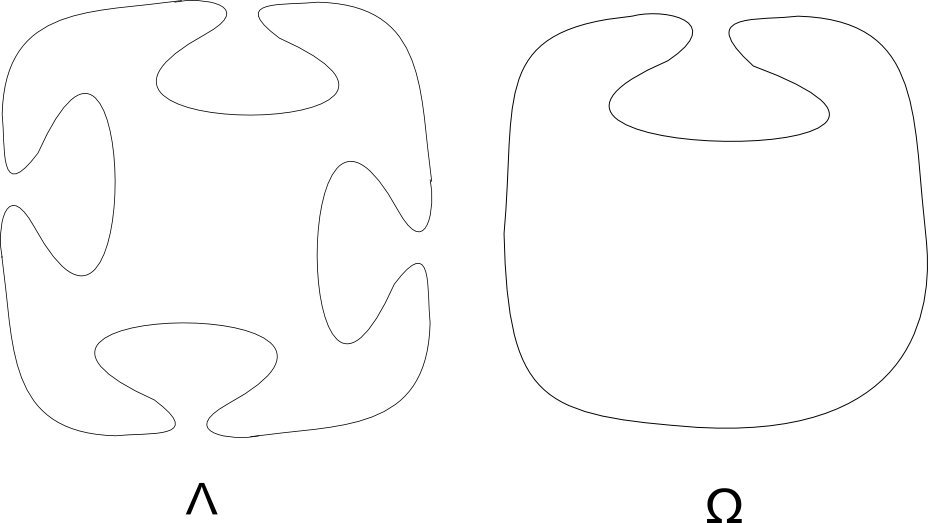}
\caption{Both domains have boundary segments that are equally negatively curved,
but the OTM is not a smooth diffeomorphism up to the boundary.}
\label{fig:3}
\end{figure}

\begin{example}
Let $\Omega=\{(x,y)\in\mathbb{R}^{2}|x>0,1<x^{2}+y^{2}<2$\} be a
half annulus and $\Lambda=\{(x,y)\in\mathbb{R}^{2}|x>0,x^2+y^{2}<1\}$
be the unit half disk. We equip them with the restricted Lebesgue
measure. Note that this equips both $\Omega$ and $\Lambda$ with
probability measures. Further note that $\Lambda$ is convex, while
$\Omega$ has a segment of curvature (the inner half circle) $-\pi$. Also observe that both boundaries are piecewise
smooth. We will show that the OTM $T:\Omega\to\Lambda$ is continuous
on the interior of $\Omega$. 

Indeed we expect the map to preserve the points radially and to ``contract''
the half annulus to the half disk by preserving the area. Such a map
would take the form 
\[
T(x,y)=\left(x\sqrt{1-\frac{1}{x^{2}+y^{2}}},y\sqrt{1-\frac{1}{x^{2}+y^{2}}}\right).
\]

It is immediate to see $T$ is a diffeomorphism between $\Omega$
and $\Lambda$ with $\det DT=1$, so that $T$ is area preserving.
Further, $T=\triangledown\varphi$ where $\varphi:\Omega\to\mathbb{R}$
is the smooth convex function given by 
\[
\varphi(x,y)=\frac{\sqrt{x^{2}+y^{2}}\sqrt{x^{2}+y^{2}-1}-\log\left(\sqrt{x^{2}+y^{2}}+\sqrt{x^{2}+y^{2}-1}\right)}{2}.
\]
But now $T$ is the gradient of a convex function, and is also area
preserving. By Brenier's theorem \cite[Theorem 2.12]{Villani}, it
is the unique OTM between $\Omega$ and $\Lambda$.
\end{example}

\section{The squareman \label{squareman}}
\label{SqManSec}

One of the characteristics of the subject of optimal transport is
that despite many deep results on existence and regularity of OTMs,
it is still very hard to explicitly compute the OTM in almost any non-trivial example. Caffarelli gave an example of a discontinuous OTM to show that without convexity of the target his regularity theory could break down \cite{Caf1992} (see also \cite[Theorem 12.3]{Villani:oldandnew}). He showed that the OTM between a disk and two half disks connected
via a sufficiently thin bridge is discontinuous. However, 
it is not exactly clear how ``thin'' the bridge should be and where
and how the discontinuity arises. In this section we will consider
an example very close to Caffarelli's example, and hopefully provide
the reader with some intuition of what the map looks like.
In particular, we will discuss where and how the discontinuity arises
in this example, and make some qualitative statements about the extent
of this discontinuity. Our computation relies on results due to Caffarelli and Figalli and we begin by recalling some of these results.

Recall that if $\mu$ and $\nu$ are two probability measures (not
necessarily uniform) supported on $\Omega$ and $\Lambda$ respectively,
then we denote by $T_1$ the optimal map which transports $\mu$
to $\nu$. Similarly, $T_2$ is the optimal map that transports
$\nu$ to $\mu$. In our setting, where $\mu,\nu$ are sufficiently
regular measures supported on domains in $\mathbb{R}^{2}$ and the
transportation cost is the quadratic cost, we have that $T_2=\nabla\phi$
for some convex function $\phi$ on $\Lambda$ and $T_1=\nabla\phi^{*}$
on $\Omega$, where $\phi^{*}$ is the Legendre transform of $\phi$
\cite[Theorem 2.12]{Villani}. Let $\gamma$ be the optimal transportation
plan between $\mu$ and $\nu$ i.e. $\gamma=(Id\times T_1)_\#\mu$.
We will need the following short lemma:

\begin{lemma}\label{shortlemma}
Assume that $\mu$ and $\nu$ are uniform measures supported on $\tilde{\Omega}$ and $\tilde{\Lambda}$ respectively, where $\tilde{\Omega}$ and $\tilde{\Lambda}$ are bounded, connected, simply connected, open domains in $\mathbb{R}^{2}$.  
Let $\tilde{\Omega}$ be convex. Then the OTM from $\tilde{\Lambda}$ to $\tilde{\Omega}$, denoted by $T$, is smooth. Further,
$T$ is a diffeomorphism between $\tilde{\Lambda}$ and an open set
$\tilde{\Omega}'\subset\tilde{\Omega}$ of full measure.
\end{lemma}

\begin{proof}
Since $\tilde{\Omega}$ is convex, Caffarelli's regularity
theory is applicable. Hence $T$ is smooth on $\tilde{\Lambda}$. By
the Monge--Amp\`ere equation, $\text{det(}DT)=\frac{|\tilde{\Omega}|}{|\tilde{\Lambda}|}=c$
for some $c>0$ almost everywhere on $\tilde{\Lambda}$. Since $T$
is smooth, it follows that $\text{det(}DT)=c>0$ everywhere.
In particular, by the inverse function theorem $T$ is a local
diffeomorphism at every point of $\tilde{\Lambda}$. To show that $T$ is a global diffeomorphism between $\tilde{\Lambda}$
and $T(\tilde{\Lambda})$, we only need to show that $T$ is injective.
This follows, for instance, from Caffarelli's result on strict 
convexity of solutions to the Monge--Amp\`ere equation above, but we also
give an elementary argument.

Indeed, assume on the contrary that there exist $x,y\in\tilde{\Lambda}$ such
that $T(x)=T(y)=z$. Pick $\epsilon>0$ such that $B_{\epsilon}(x)\cap B_{\epsilon}(y)=\emptyset$
and $T$ is a diffeomorphism when restricted separately to both
$B_{\epsilon}(x)$ and $B_{\epsilon}(y)$. Let $A=T(B_{\epsilon}(x))\cap T(B_{\epsilon}(y))$.
Since $A$ is non-empty and open, it must have positive measure. But
then, for every $a\in A$ the set $\{b\in\tilde{\Lambda}:(a,b)\in\text{supp}(\gamma)\}$
contains at least two elements - one from $B_{\epsilon}(x)$ and one
from $B_{\epsilon}(y)$, where $\gamma$ is the optimal transportation plan between $\mu$ and $\nu$. This means that $\gamma$ is not a Monge
map since it sends every element in $A$ to at least two locations,
which contradicts Brenier's theorem. It follows that $T$ is
a global diffeomorphism between $\tilde{\Lambda}$ and $T(\tilde{\Lambda})=\tilde{\Omega}'$.

Note that $\tilde{\Omega}'\subset\tilde{\Omega}$.
Since $T$ is optimal, $|\tilde{\Omega}'|=|T(\tilde{\Lambda})|=|\tilde{\Lambda}|=|\tilde{\Omega}|$
and therefore, $\tilde{\Omega}'$ is a set of full measure. This completes
the proof of the lemma.\end{proof}

Finally, in our example below, we will need two additional properties, which we state now: \\

\noindent Property \textbf{A}: Restrictions of optimal maps are still optimal
between the restricted domain and its image. \cite[Theorem 4.6]{Villani:oldandnew}\\

\noindent Property \textbf{B}: If the optimal map between $\Omega$ and $\Lambda$
is of the form $\triangledown\phi$, then the set $\{x\in\mathbb{R}^{2}|\partial\phi(x)\cap\overline{\Lambda}$
contains a segment$\}$ is empty. \cite[Proposition 3.2]{Figalli}

\subsection{Explicit Example}

We now introduce a specific example we refer to as the squareman.
Let $\mu$ be the uniform probability measure on a rectangle $\Omega$
with sides $|A_{1}B_{1}|=a$ and $|G_{1}A_{1}|=b$, and $\nu$ be
the uniform probability measure on $\Lambda$, which is made up of
two rectangles - a rectangle $\Lambda_{2}$ with sides $|AB|=a$ and
$|GA|=b$, and another rectangle $\Lambda_{1}$ on top of it with
sides $|CD|=c$ and $|ED|=d$, where $d<a$ (see Figure \ref{fig:4}). Note that
since the optimal map is invariant under translations of $\Omega$
and $\Lambda$, it does not depend on the relative positions of the
domains with respect to each other. Our example is similar to the
one by Caffarelli: indeed, if we work with rectangles instead of disks,
then the example by Caffarelli transforms to transporting a rectangle
to an H-shape figure. Due to symmetry, we can divide these figures
into 4 different symmetric parts and look at the OTM for each one.
But this is exactly the example we are considering.

Since $\Omega$ is convex, it follows from Lemma \ref{shortlemma} that the map
$T_2$ is a diffeomorphism onto its image, which is a set of full
measure in $\Omega$. Further, by Proposition \ref{App4}, it follows that there exists a continuous extension of $T_2$ from $\overline{\Lambda}$ to $\overline{\Omega}$. However, since $\Lambda$ is not convex, we
cannot make any such claims about $T_1$. In fact, we will show
below that $T_1$ is discontinuous, and further, give a qualitative
statement of how discontinuous it is. Note that since $\partial\Lambda$ does not contain any connected subset with total signed curvature less than $-\pi/2$, this shows that the condition in Theorem \ref{mainthm} is not necessary. 

Note that Brenier's theorem and Caffarelli's regularity theory are
applicable only to the interiors of $\Omega$ and $\Lambda$. However,
in this particular example, we will take advantage of the fact that
the boundaries of the two domains are parts of straight segments.
Using this, we will be able to identify where parts of the boundary $\partial\Lambda$ are mapped by $T_2$, which will 
give us useful information about the discontinuity
of the map $T_1=\nabla\phi^{*}$. Indeed, if we knew that two
points $x_{1}\in EF$ and $x_{2}\in GF$ are mapped to an interior
point $x\in\Omega$, and since $T_2$ is continuous, then $\partial\phi^{*}(x)$
would contain both $x_{1}$ and $x_{2}$, and in particular, the extent
of discontinuity of $T_1$ at $x$ would be at least the distance
between $x_{1}$ and $x_{2}$. 

\begin{figure}
\includegraphics[scale =.7]{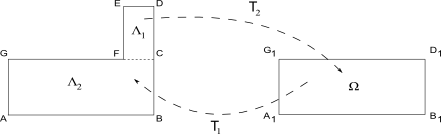}
\caption{The squareman example. The optimal map $T_{1}$ pictured will be seen to be discontinous.}\label{fig:4}
\end{figure}

In the following four steps we will determine how the map $T_2$
behaves on the boundary of $\Lambda$ and this will also give us information
about the behaviour of $T_1$. In particular, we will justify Figure
\ref{fig:5}, in which same-colored segments are mapped to same-colored segments.
The 4 steps are as follows:\\

\noindent
$\emph{Step 1:}$ $T_2(AB)=A_{1}B_{1}$ and $T_2(BD)=B_{1}D_{1}$
and both restrictions are homeomorphisms. Furthermore, $T_2(AG)\subset A_{1}G_{1}$
and $T_2(DE)\subset D_{1}G_{1}$.

\noindent
$\emph{Step 2:}$ We have that either $T_2(G)=G_{1}$ or $T_2(E)=G_{1}$.
In particular, we can assume without any loss of generality that $T_2(G)=G_{1}$, so $T_2(AG)=A_{1}G_{1}$
homeomorphically. We will further show that $T_2(EF)$ lies completely
in $\Omega$ except for, of course, $T_2(E)=E_{1}$ which is on
the boundary.

\noindent
$\emph{Step 3:}$ For some $E'\subset GF$ we have that $T_2(GE')=G_{1}E_{1}$
homeomorphically.

\noindent
$\emph{Step 4:}$ $T_2(E'F)=T_2(EF)$.\\


\noindent
$\emph{Step 1:}$ $T_2(AB)=A_{1}B_{1}$ and $T_2(BD)=B_{1}D_{1}$
and both restrictions are homeomorphisms. Furthermore, $T_2(AG)\subset A_{1}G_{1}$
and $T_2(DE)\subset D_{1}G_{1}$.\\

To get the desired information on $\partial\Lambda$, we are going
to use what will henceforth be referred to as $\emph{the reflection principle}$.
More precisely, we reflect $\Lambda$ with respect to $AB$ to get
a domain of twice the area, $R_{AB}(\Lambda)$ and reflect $\Omega$
with repect to $A_{1}B_{1}$ to get a domain $R_{A_{1}B_{1}}(\Omega)$. See Figure \ref{fig:6}.

\begin{figure}
\includegraphics[scale = .7]{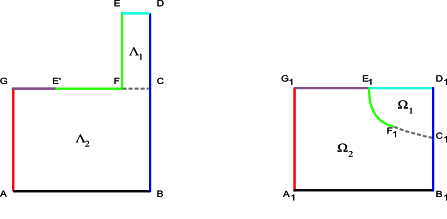}
\caption{How the OTM behaves on the boundary of the domains.}\label{fig:5}
\end{figure}

Let $\mu'$ and $\nu'$ be uniform probability measures on $R_{A_{1}B_{1}}(\Omega)$
and $R_{AB}(\Lambda)$ respectively with respective uniform densities
$f'$ and $g'$. As before, since $R_{A_{1}B_{1}}(\Omega)$ is convex,
it follows from Caffarelli's regularity theory that the ($\nu$-almost
everywhere) unique, optimal map $T'_{2}$ from $R_{AB}(\Lambda)$
to $R_{A_{1}B_{1}}(\Omega)$ is smooth. We claim $T'_{2}(\Lambda)\subset\Omega$.\textbf{ }

For simplicity let $AB$ and $A_{1}B_{1}$ lie on the $x$-axis and
for any $z\in\mathbb{R}^{2}$ let $\overline{z}$ denote its reflection
with respect to the $x$-axis. Due to symmetry, the map $T(x)=\overline{T'_{2}(\overline{x})}$
has the same cost as $T'_{2}$ and by the uniqueness of the OTM,
we get $T'_{2}(x)=\overline{T'_{2}(\overline{x})}$ for a.e. $x$.
But now if $T'_{2}(x)=y$ and $x$ and $y$ are on different sides
of the $x$-axis, we know $T'_{2}(\overline{x})=\overline{y}$. It is immediately
checked that this violates the cyclic monotonicity condition for $x$
and $\overline{x}$. In particular for almost all $x\in\Lambda$ we get
that $T'_{2}(x)\in\Omega$ and since $T'_{2}$ is smooth we conclude
that $T'_{2}(\Lambda)\subset\Omega$. Also note $T'_{2}(AB)\subset A_{1}B_{1}$
since otherwise by the continuity of $T'_{2}$ we could find a ball
around a point on $AB$ whose image under $T'_{2}$ lies completely
above or below the $x$-axis, which as we explained above violates
cyclic monotonicity. Note that there are two key conditions in the
use of the reflection principle. First we need to reflect along a
straight segment. Second we need one of the domains after reflection
to be convex.

From here, since optimality is inherited by restriction (Property
\textbf{A}), and since Brenier's theorem also guarantees almost-everywhere
uniqueness of the optimal map, $T'_{2}|_{\Lambda}$ coincides with
$T_2$ almost everywhere. Since $T'_{2}$ is smooth, this gives
us a smooth extension of $T_2$ to the interior of the segment
$AB$ and $T_2(AB)\subset A_{1}B_{1}$. Note that Proposition \ref{App4} already gives us a continuous extension of $T_2$ to the entire boundary $\partial\Lambda$. However, by using the reflection principle here, we have a smooth extension of $T_2$ to the interior of the segment $AB$, and more importantly, we get information about the image of this extension on the interior of the segment $AB$. Similarly, we get an extension
of $T_2$ to the interior of $BD$ and $T_2(BD)\subset B_{1}D_{1}$.
We will use the same notation for $T_2$ and its continuous extensions
to (parts of) $\partial\Lambda$.

Further, observe we can also use the reflection principle again and
reflect $R_{AB}(\Lambda)$ with respect to the line containing $BD$
to get $R'_{BD}(\Lambda)$ and reflect $R_{A_{1}B_{1}}(\Omega)$ with
respect to the line containing $B_{1}D_{1}$ to get $R'_{B_{1}D_{1}}(\Omega)$.
Now the newly obtained domains are 4 times the size of the original
ones and $R'_{B_{1}D_{1}}(\Omega)$ is still convex. The motivation
for doing so is to include the points $B$ and $B_{1}$ in the interiors
of the domains $R'_{BD}(\Lambda)$ and $R'_{B_{1}D_{1}}(\Omega)$
respectively. Exactly as above due to symmetry and smoothness of the
optimal map, it follows that the optimal map must send $B$ to $B_{1}$.
Therefore using Lemma \ref{shortlemma} we have that $T_2$ maps the half-open
segment $BA$ (with $B$ included) injectively to a (possibly strict) subset of $B_{1}A_{1}$ (with
$B$ mapping to $B_{1}$)

\begin{figure}
\includegraphics[scale = .7]{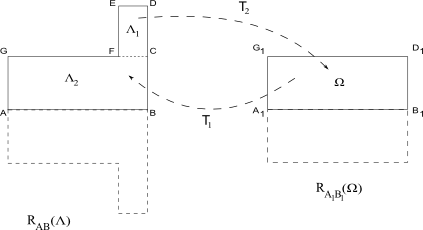}
\caption{The \emph{reflection principle}: We reflect both domains as to make the
boundary part of the interior and use the optimal map is preserved
under restriction.}\label{fig:6}
\end{figure}

Similarly, we can reflect $R_{AB}(\Lambda)$ with respect to the line
containing $AG$ to get $R''_{AG}(\Lambda)$ and reflect $R_{A_{1}B_{1}}(\Omega)$
with respect to the line containing $A_{1}G_{1}$ to get $R''_{A_{1}G_{1}}(\Omega)$.
Using exactly the same arguments as in the preceding paragraph, it
follows that $T_2$ extends to a continuous, injective map on
the entire closed segment $AB$, $T_2(A)=A_{1},T_2(B)=B_{1}$
and $T_2(AB)\subset A_{1}B_{1}$. In fact, since $T_2(A)=A_{1}$
and $T_2(B)=B_{1}$, we have that $T_2(AB)=A_{1}B_{1}$ so
that $T_2$ is a homeomorphism between the closed segments $AB$
and $A_{1}B_{1}$. 

Note that symmetry considerations in the optimal map from $R''_{AG}(\Lambda)$
to $R''_{A_{1}G_{1}}(\Omega)$ lead to the conclusion that the interior
of the segment $AG$ must be mapped to a part of the interior of the
segment $A_{1}G_{1}$. Since we cannot perform any reflection of the
form considered earlier to include $G$ in the interior of the reflected
domain, we cannot claim that $T_2$ can be extended to a continuous
map till $G$ and, in particular, that $T_2(G)=G_{1}$. In fact,
as we will end up showing $T_2(G)$ might be distinct from $G_{1}$.

It is clear now, using exactly the same arguments as above, that $T_2(D)=D_{1}$
and in fact, that $T_2$ is a homeomorphism between the closed
segments $BD$ and $B_{1}D_{1}$. Similarly, it also follows that
$T_2$ extends smoothly to the interior of the segment $DE$ and
sends it to part of the interior of the segment $D_{1}G_{1}$. \\

\noindent
$\emph{Step 2:}$ We have that either $T_2(G)=G_{1}$ or $T_2(E)=G_{1}$.\\

After proving Step 2, we may thus assume (without loss of generality) that $T_2(G)=G_{1}$, so
$T_2(AG)=A_{1}G_{1}$ homeomorphically. We will further show that
$T_2(EF)$ lies completely in $\Omega$ except for, of course,
$T_2(E)=E_{1}$ which is on the boundary.

To check the above first note that $T_2(\overline{\Lambda})\subset\overline{\Omega}$
is of full measure and is compact, so it must be that $T_2(\overline{\Lambda})=\overline{\Omega}$.
Furthermore we know $T_2(\Lambda)\subset\Omega$, so it must be
that $\partial\Omega\subset T_2(\partial\Lambda)$. Since in Step
1 we determined where $\partial\Lambda$ is mapped except for the
segments $GF$ and $EF$ we must have that 
\begin{equation}\label{eq:3.1}
G_{1}\in T_2(GF)\text{\quad or \quad} G_{1}\in T_2(EF).
\end{equation}
Now we will study $T_2(EF)$. Set $\alpha=T_2(EF\cup FC)$. Note $\alpha$ is a continuous curve
from $E_{1}$ to $T_2(C)=C_{1}$. Then $\alpha\cup C_{1}D_{1}\cup D_{1}E_{1}$
is a continuous loop. Note it is simple since if a point $x$ is an
intersection point, then $\partial\phi^{*}(x)$ would contain a segment
in $\overline{\Lambda_{1}}\subset\overline{\Lambda}$, which contradicts
(Property $\textbf{B}$). Hence by the Jordan curve theorem the loop
bounds some open set $\Omega_{1}\subset\Omega$, and we will refer to this loop as $\partial\Omega_{1}$. Note that neither $T_{2}(\Lambda_{2})$ nor $T_{2}(\Lambda_{1})$ can intersect $\partial\Omega_{1}\cap\Omega$ (Property $\textbf{B}$). Also, both $T_{2}(\Lambda_{2})$ and $T_{2}(\Lambda_{1})$ are path connected, since $\Lambda_{1}$ and $\Lambda_{2}$ are path connected and $T_2$ is continuous. In particular, $T_2(\Lambda_{1})$ is either completely contained inside $\Omega_{1}$ or completely contained inside $\Omega\backslash\overline{\Omega_{1}}$, and a similar statement also holds for $T_{2}(\Lambda_{2})$. Finally, note that by Step 1, the intersection of a ball of sufficiently small radius centered at $D_1$ with $\overline{\Omega}$ is contained inside $\overline{\Omega_1}$.

We claim that $T_{2}(\Lambda_{1})\subset\Omega_{1}$. Indeed, suppose this is not so. Then, we must have $T_2(\Lambda_{1})$ is contained in $\Omega\backslash\overline{\Omega_{1}}$ as noted above. We will show that this leads to a contradiction. To see this, choose a continuous path $\gamma$ in $\Lambda_{1}$ with one endpoint at $D$. 
Since $T_2(D)=D_1$, and the intersection of a sufficiently small ball centered at $D_1$ with $\overline{\Omega}$ is contained inside $\overline{\Omega_1}$, it follows that there exist points in $\gamma$ which must be mapped to $\Omega_1$ by $T_2$, so that we cannot have $T_2(\Lambda_{1})\subset\Omega\backslash\overline{\Omega_{1}}$. Therefore, $T_2(\Lambda_{1})\subset\Omega_{1}$ and in fact, $T_2(\Lambda_{1})$ is of full measure in $\Omega_{1}$ since $T_2(\Lambda)\subset\Omega$ is of full measure in $\Omega$. Since the restriction of an optimal map (Property \textbf{A}) is
optimal we get that $T_{2|\Lambda_{1}}\colon\Lambda_{1}\to\Omega_{1}$
is the optimal map between the two domains.

Now we are in a position to study $T_2(EF)$. Let $T_2(K)=K_{1}$
be the point on $T_2(EF)$ satisfying the following two properties:
(a) it lies on $G_{1}E_{1}$ (b) it has the least distance to $G_{1}$
among all the points on $EF$ whose images under $T_2$ lies on
$G_{1}E_{1}$. We will show that $K_{1}=E_{1}$. 

Observe first that $K_{1}E_{1}\subset T_2(EF)$. Indeed if this
was not the case then for some $M\in EK$ we would have that $T_2(M)\in\Omega$,
but then the pair of points $K,M$ would violate the cyclical monotonicity
condition. Note that even though these points are on the boundary
of $\Lambda$, the continuity of $T_2$ allows us to extend the
cyclical monotonicity condition to these points. Hence $K_{1}E_{1}\subset T_2(EF)$. 

Now, we reflect $\Omega_{1}$ with respect to $K_{1}D_{1}$ and reflect
$\Lambda_{1}$ with respect to $ED$. The reflection of $\Lambda_{1}$
is a convex domain, and therefore, we can use the reflection principle
as before, from which we can conclude that $T_1(K_{1}D_{1})\subset ED$.
Note here we are using the fact that $\overline{\Lambda_{1}}$ is convex to be
able to define $T_1$ everywhere on $\overline{\Omega_{1}}$ (Proposition \ref{App4}). However, we
already knew that $T_2(ED)=E_{1}D_{1}$, so we conclude $K_{1}D_{1}\subset E_{1}D_{1}$,
so $K_{1}=E_{1}$ as desired. 

In particular, we get that $T_2(EF)\cap E_{1}G_{1}=E_{1}$. Applying
the same argument for $T_2(GF)$ we conclude that 
\begin{equation}
T_2(EF)\cap E_{1}G_{1}=E_{1}\:\;\text{ and }\:\; T_2(GF)\cap\overline{G_{1}}G_{1}=\overline{G_{1}}\label{eq:3.2}
\end{equation}
where $\overline{G_{1}}=T_2(G)$. In particular, \eqref{eq:3.1}
and \eqref{eq:3.2} imply that $T_2(E)=E_{1}=G_{1}$ or $T_2(G)=G_{1}$.
Due to symmetry between $E$ and $G$, we can assume without loss
of generality that $T_2(G)=G_{1}$. Further, if $x\in T_2(EF)\cap(\partial\Omega\backslash G_{1}E_{1})$
then since we have covered $\partial\Omega\backslash G_{1}E_{1}$
by boundary segments of $\partial\Lambda\backslash EF \cup GF$ we would get
that $\partial\phi^{*}(x)$ contains a segment in $\overline{\Lambda}$,
which contradicts (Property \textbf{B}). Hence $T_2(EF)$ is in $\Omega$
except for $T_2(E)=E_{1}\in\partial\Omega$. Note that here we
used the fact that the only points of $\Lambda_{1}$ that can be mapped
to the same point are pairs of points on $EF$ and $GF$ since the
segment connecting them intersects $\overline{\Lambda}$ at isolated
points and not segments (Property \textbf{B})\\

\noindent
$\emph{Step 3:}$ For some $E'\subset GF$ we have that $T_2(GE')=G_{1}E_{1}$
homeomorphically.\\

As we explained at the beginning of Step 2 we must have that $\partial\Omega\subset T_2(\partial\Lambda)$.
The only part of $\partial\Lambda$ whose position under $T_2$
we haven't determined yet is $GF$. At the same time we know $G_{1}E_{1}\cap T_2(\partial\Lambda\backslash GF)=\emptyset$,
so we conclude that $G_{1}E_{1}\subset T_2(GF)$. Let $E'\in GF$
be such that $T_2(E')=E_{1}$. It is clear now that by Property
\textbf{B}, $T_2(GE')$ is simple. Further, note that if some
$M\in E'F$ maps to the $G_{1}E_{1}$ then the pair $M,E'$ would
violate the cyclical monotonicity condition. Hence since $G_{1}E_{1}\subset T_2(GF)=T_2(GE'\cup E'F)$
we must have that $G_{1}E_{1}\subset T_2(GE')$. But now $T_2(GE')$
is simple with endpoints $G_{1}$ and $E_{1}$ and it contains $G_{1}E_{1}$,
so it must be that $T_2(GE')=G_{1}E_{1}$ as desired.\\

\noindent
$\emph{Step 4:}$ $T_2(E'F)=T_2(EF)$.\\

Set $\alpha_{1}=T_2(EF)$ and $\alpha_{2}=T_2(FC)$ and note
that both are simple curves (Property
\textbf{B}). Following the notation from above we
have that $\alpha=T_2(EF\cup FC)=\alpha_{1}\cup\alpha_{2}$. Now
$\alpha\cup C_{1}B_{1}\cup B_{1}A_{1}\cup A_{1}G_{1}\cup G_{1}E_{1}$
is a Jordan curve, so it bounds some $\Omega_{2}$. In particular
we have that $\Omega_{1}\cup\Omega_{2}\cup\mathring{\alpha}=\Omega$
where $\mathring{\alpha}$ is the curve without its endpoints. In
Step 2 we showed that $T_2(\Lambda_{1})$ is of full measure in
$\Omega_{1}$. Hence $T_2(\Lambda_{2})\subset\Omega_{2}$ and is
of full measure in $\Omega_{2}$. Therefore $T_{2|\Lambda_{2}}:\Lambda_{2}\to\Omega_{2}$
restricts to an optimal map by Property \textbf{A}. $T_{1|\Omega_{2}}$
is its inverse, so it is optimal as well. 

But now $\Lambda_{2}$ is
convex, so $T_{1|\Omega_{2}}$ is smooth on $\Omega_{2}$ by Caffarelli's
regularity theory and extends continuously to the boundary. Now $T_{1|\Omega_{2}}(\overline{\Omega_{2}})\subset\overline{\Lambda_{2}}$
is compact and of full measure, so $T_1(\overline{\Omega_{2}})=\overline{\Lambda_{2}}$.
Since $T_1(\Omega_{2})\subset\Lambda_{2}$ we conclude that $\partial\Lambda_{2}\subset T_1(\partial\Omega_{2})$.
We know that $T_2$ sends $\partial\Lambda_{2}\backslash E'F$
to $\partial\Omega_{2}\backslash\alpha_{1}$ homeomorphically by the previous three steps. This means $T_1(\partial\Omega_{2}\backslash\alpha_{1})=\partial\Lambda_{2}\backslash E'F$,
so we must have that $E'F\subset T_1(\alpha_{1})$. But note $T_1(\alpha_{1})$
is a simple curve (by Property $\textbf{B}$) that contains $E'F$
and has endpoints $E'$ and $F$. Hence it must be that $T_1(\alpha_{1})=E'F$
homeomorphically, so we conclude that $T_2(E'F)=\alpha_{1}=T_2(EF)$
as desired.

Thus we have completely determined $T_2(\partial\Lambda)$ and
thus we have justified the picture in Figure \ref{fig:5}. Further, we have determined
that the set of discontinuity of $T_1$ is exactly the curve $\alpha_{1}$
and for every $x\in\alpha_{1}$, the subdifferential $\partial\phi^{*}(x)$
is a segment that connects the preimages of $x$ on $E'F$ and $EF$.
In particular, as $x$ moves from $T_1(F)$ to $E_{1}$ this segment
grows and reaches $|EE'|$ as $x$ gets to $E_{1}$.

\appendix

\section{}\label{sec:app}

\begin{prop}
\label{App1}
Assume $\Lambda$ is a bounded simply connected domain
with $C^{2}$ boundary. Then there exists an $\epsilon_{o}>0$ such that for all $0<\epsilon<\epsilon_{o}$, 
$\Gamma_{\epsilon}=\{x\in\Lambda: \h{dist}(x,\partial\Lambda)=\epsilon\}$
is a $C^{1}$ simple curve with a diffeomorphism $f_{\epsilon}:\partial\Lambda\to\Gamma_{\epsilon}$
such that $f_{\epsilon}(x)-x$ is normal to both curves at the points
$x$ and $f_{\epsilon}(x)$ and $|f_{\epsilon}(x)-x|=\epsilon$.
\end{prop}

\begin{proof}
Define $v:\partial\Lambda\to S^{1}$ to be the inward
pointing unit normal vector field. Next define $F:\partial\Lambda\to\Lambda$
by $F(x)=x+\epsilon v(x)$. This map is well defined for $\epsilon$
small, so that $Im(F)\subset\Lambda$. In fact we will show $F$ is
the desired diffeomorphic map $f_{\epsilon}$.

Note that for any $x\in\partial\Lambda$ there exists $\epsilon_{x}=\sup\{\overline{\epsilon}>0:\overline{B_{\overline{\epsilon}}(x+\overline{\epsilon}v(x))}\cap\partial\Lambda=\{x\}\}$. We claim that there exists an $\epsilon_{o}>0$ such that $\epsilon_{x} \geq \epsilon_{o}$ for all $x\in\partial\Lambda$. Indeed, this is true because $\partial\Lambda$ is assumed to be $C^{2}$, so that its radius of curvature is a continuous, strictly positive function on $\partial\Lambda$, which is assumed to be compact. By taking $\epsilon_{o}$ to be smaller than the positive lower bound for the radius of curvature, the proof of the claim is complete.
Therefore, for $0<\epsilon<\epsilon_{o}$, we get $\epsilon=|F(x)-x|= \h{dist}(F(x),\partial\Lambda)$, for any $x\in\partial\Lambda$, so that
in particular $Im(F)\subset\Gamma_{\epsilon}$. Furthermore if $y\in\Gamma_{\epsilon}$
then there exists $x\in\partial\Lambda$ with $|x-y|=\epsilon$, so
in particular $y=F(x)$. Hence $Im(F)=\Gamma_{\epsilon}$. Next note
that $F$ is clearly 1-1 since $\overline{B_{\epsilon}(F(x))}\cap\partial\Lambda=\{x\}$
for every $x\in\partial\Lambda$. Hence $F$ is a bijection.

Finally, note that since $\partial\Lambda$ is $C^{2}$ then $v$ is
$C^{1}$ and so $F$ is $C^{1}$. Hence $Im(F)=\Gamma_{\epsilon}$
is a $C^{1}$ curve which is diffeomorphic to $\partial\Lambda$. Also
by definition $x-F(x)$ is normal to $\partial\Lambda$ at $x$ and
since $\h{dist}(x,\Gamma_{\epsilon})=\epsilon=|x-F(x)|$ we also have
that $x-F(x)$ is normal to $\Gamma_{\epsilon}$ at $F(x)$. This completes
the proof of the proposition.\end{proof}

\begin{prop}\label{App2}Assume $\gamma:[0,1]\to\mathbb{R}^{2}$
is a $C^{2}$ simple (non-closed), regular curve. Then there exists $t\in(0,1)$ such that
$\gamma'(t)$ is parallel to $\gamma(1)-\gamma(0)$.
\end{prop}
\begin{proof}
Let $\gamma(s) = (\gamma_{1}(s),\gamma_{2}(s))$. By Cauchy's mean value theorem, there exists some $t\in(0,1)$ such that 
$(\gamma_{1}(1)-\gamma_{1}(0))\gamma_{2}'(t)=(\gamma_{2}(1)-\gamma_{2}(0))\gamma_{1}'(t)$. 
Without loss of generality, we can assume that $\gamma_{2}(1)-\gamma_{2}(0)\neq 0$. 
Then, if  $\gamma_{2}'(t)=0$, we get that $\gamma_{1}'(t)=0$, which contradicts the regularity assumption. Dividing by $(\gamma_{2}(1)-\gamma_{2}(0))\gamma_{2}'(t)$, we have the result. 
\end{proof}

\begin{prop}\label{App3}
Assume $\Lambda$ is a bounded simply connected
domain with $C^{2}$ boundary. Let $\epsilon>0$ be such that $\Gamma_{\epsilon}=\{x\in\Lambda: \h{dist}(x,\partial\Lambda)=\epsilon\}$
is a $C^{1}$ simple curve and $f_{\epsilon}:\partial\Lambda\to\Gamma_{\epsilon}$
be a diffeomorphism such that $f_{\epsilon}(x)-x$ is normal to both
curves at the points $x$ and $f_{\epsilon}(x)$ and $|f_{\epsilon}(x)-x|=\epsilon$
(guaranteed to exist by Proposition \ref{App1}). 
Now let $\gamma:[0,1]\to\mathbb{R}^{2}$ be a simple $C^{1}$ closed curve that lies in $\Lambda$ and contains $\Gamma_{\epsilon}$ in its interior in the sense that every continuous path between $\Gamma_{\epsilon}$ and $\partial\Lambda$ must intersect $\gamma$. Also assume that $\gamma$ intersects neither $\partial\Lambda$ nor $\Gamma_{\epsilon}$. Then for all pairs $C,D\in\partial\Lambda$ of distinct points, if $I\subset\partial\Lambda$ is the subset of $\partial\Lambda$ that connects $C$ and $D$ as we move clockwise around $\partial\Lambda$ there exists a connected subset of $\gamma$ that lies completely in the closed figure $\Theta_{CD}$ bounded by the segments $[C,C_{2}],[D,D_{2}]$ and by $I$ and $f_{\epsilon}(I)$ such that it has one endpoint on each of $[C,C_{2}]$ and $[D,D_{2}]$, where $C_{2}=f_{\epsilon}(C)$and $D_{2}=f_{\epsilon}(D)$.
\end{prop}
This proposition is illustrated in Figure 2, where the path between $P$ and $Q$ is the connected subset of $\gamma$ which the proposition guarantees. 
\begin{proof}
We prove this by contradiction. Without loss of generality,
we can assume that $\gamma(0)$ lies outside $\Theta_{CD}$. Let $t_{1}=\inf\{t\colon\gamma(t)\in\Theta_{CD}\}$
and $t_{2}=\sup\{t\colon\gamma(t)\in\Theta_{CD}\}$. By the compactness
of the unit interval, the continuity of $\gamma$, and the closedness
of $\Theta_{CD}$, the infimum and supremum are attained.

Consider $S=\gamma[t_{1},t_{2}]\cap\Theta_{CD}$. Since we are assuming
that there does not exist a connected subset of $\gamma$ that lies
completely in the closed figure $\Theta_{CD}$ such that it has one
endpoint on each of $[C,C_{2}]$ and $[D,D_{2}]$, it follows that
there are two, mutually exclusive types of path connected components
of $S$ - those that intersect $[CC_{2}]$ and those that intersect
$[DD_{2}]$. Let $A'$ denote the union of all the path connected
components of $S$ that intersect $[CC_{2}]$ and $B'$ denote the
union of those that intersect $[DD_{2}]$. Then, $A'$ and $B'$ are
disjoint, and we claim that they are also compact subsets of $\mathbb{R}^{2}$. 

First, let us see how the compactness of $A'$ and $B'$ finishes
the proof. Let $\Delta'_{CD}$ denote the complement of $\Theta_{CD}$
in the region enclosed between $\partial\Lambda$ and $\Gamma_{\epsilon}$.
Let $\Delta_{CD}=\Delta'_{CD}\cup[CC_{2}]\cup[DD_{2}]$. Note that
$\Delta_{CD}$ is a closed and bounded, hence compact subset of $\mathbb{R}^{2}$.
Finally, let $A=A'\cup\Delta_{CD}$ and $B=B'\cup\Delta_{CD}$. Then,
the compactness of $A',B'$ implies that $A,B$ are compact and further,
for $E$ in the interior of $I$ and $E_{2}=f_{\epsilon}(E)$, we
get that $E,E_{2}$ are not separated by $A$ or $B$ in the sense
that they lie in the same path connected (and hence, connected) component
of $\mathbb{R}^{2}\backslash A$ and of $\mathbb{R}^{2}\backslash B$.
To see that $E,E_{2}$ are not separated by $A$, we begin by noting
that the compactness of $A'$ implies that the distance between $A'$
and $[D,D_{2}]$ is always greater than some $\epsilon>0$. Then,
we can go from $E,E_{2}$ in $\mathbb{R}^{2}\backslash A$ by travelling
along $I$ towards $D$ until we are at a distance $\epsilon/2$ away
from $D$, then moving parallel to $[D,D_{2}]$ until we hit $f_{\epsilon}(I)$,
and finally, moving along $f_{\epsilon}(I)$ to $E_{2}$. Here, we
use the fact that $A'$ does not intersect $I$ or $f_{\epsilon}(I)$.
A similar argument shows that $E,E_{2}$ are not separated by $B$. 

Now, we recall Janiszewski's theorem \cite[Ap.3.2]{Dieudonne}, which
says that if $A,B$ are compact subsets of $\mathbb{R}^{2}$ such
that $A\cap B$ is connected, and $E,E_{2}\in\mathbb{R}^{2}\backslash A\cup B$
such that neither $A$ nor $B$ separates $E,E_{2}$, then $A\cup B$
also does not separate $E,E_{2}$. This implies that $\gamma$ does
not separate $E,E_{2}$ i.e. $E,E_{2}$ lie in the same connected
component of $\mathbb{R}^{2}\backslash\gamma$ (and hence, the same
path connected component of the open subset $\mathbb{R}^{2}\backslash\gamma$
of $\mathbb{R}^{2}$). But this contradicts the hypothesis that $\Gamma_{\epsilon}$
lies in the interior of $\gamma$. 

So, to finish the proof, we only need to show that $A',B'$ are compact.
We do this only for $A'$, the proof for $B'$ being similar. Since
$A'$ is a bounded subset of $\mathbb{R}^{2}$, we only need to show
that it is closed in $\mathbb{R}^{2}$. Let $p_{n}$ be a sequence
of points in $A'$ such that $p_{n}\rightarrow p$. There exists a
unique sequence $\{a_{n}\}\in[t_{1},t_{2}]$ such that $p_{n}=\gamma(a_{n})$.
By the compactness of $[t_{1},t_{2}]$, we can, after possibly passing
to a subsequence, assume that $a_{n}\rightarrow a\in[t_{1,},t_{2}]$. After possibly passing to another subsequence, we can further assume that either $a_{n} < a$ or $a_{n} > a$ (since the case where $a_{n}$ is eventually $a$ is trivial). In the subsequent discussion, we will assume that $a_{n} < a$. The case $a_{n} > a$ is treated similarly.  

By the continuity of $\gamma$, it follows that $p=\gamma(a)$.
Since $a\in[t_{1},t_{2}]$, we have that either $p\in A'$ or $p\in B'$.
If $p\in A'$, then we are done, so suppose that $p\in B'$. Since
$\gamma[a_{n},a]$ is path connected but $\gamma(a_{n})$ and $\gamma(a)$
are in different path connected components of $S$, it follows that
$\gamma[a_{n},a]$ must leave $S$ through $[C,C_{2}]$. Let $a_{n}<b_{n}<a$
be such that $\gamma(b_{n})\in[C,C_{2}]$. Note that we can always
find such a $b_{n}$ by the previous remark. But then, $b_{n}\rightarrow a$,
so that from the continuity of $\gamma$ and the closedness of $[C,C_{2}]$
we get $\gamma(a)=\lim_{n}\gamma(b_{n})\in[C,C_{2}]$, which contradicts
that $\gamma(a)\in B'$. This completes the proof. 
\end{proof}
\begin{prop}\label{App4} Let $\Omega$ and $\Lambda$ be bounded, connected, simply connected open domains in $\mathbb{R}^{2}$, equipped with the uniform measures $\mu$ and $\nu$. Assume that $\overline{\Omega}$ 
convex. Then, the OTM $T\colon\Lambda\to\Omega$ admits a single-valued, continuous extension to 
$\overline{\Lambda}$.
\end{prop}
\begin{proof}
By considering $\nu$ as a measure on all of $\RR^2$ supported
on $\Lambda$, Brenier's theorem furnishes a globally Lipschitz convex function $\varphi\colon\mathbb{R}^{2}\to\mathbb{R}$ such that $T=\nabla\varphi$ on $\Lambda$ (the equality holds everywhere, instead of just almost everywhere, because of Caffarelli's regularity theory), and
$\partial\varphi(\mathbb{R}^{2})\subset\overline{\Omega}$ 
since $\overline{\Omega}$ is convex
\cite[Lemma 1(b)]{Caf1992}. 
For $x\in\RR^2$, the set $\partial\varphi(x)$ is convex \
\cite[p. 215]{Rock},
and so if it is not a singleton it contains
a segment that is contained 
in the convex set $\overline{\Omega}$. 
However, this contradicts
Property \textbf{B} stated in \S\ref{SqManSec}. In particular, $\varphi$ is differentiable on $\mathbb{R}^{2}$
and therefore is $C^1$ \cite[Theorem 25.5]{Rock}, implying the statement.
\end{proof}


\end{document}